\documentclass[10pt]{amsart}

\usepackage[colorlinks]{hyperref}
\usepackage{color,graphicx,shortvrb}
\usepackage[latin 1]{inputenc}

\usepackage[active]{srcltx} %SRC Specials for DVI search

\usepackage{enumerate}
\usepackage{amssymb}

\newtheorem{theorem}{Theorem}[section]

\newtheorem{corollary}[theorem]{Corollary}

\newtheorem{proposition}[theorem]{Proposition}

\def\11{\textbf{$1$}}

\DeclareGraphicsExtensions{.jpg,.pdf,.png,.eps}

\begin{document}

\title[Tingley's problem for von Neumann algebras]{On the extension of isometries between the unit spheres of von Neumann algebras}

\author[F.J. Fern\'{a}ndez-Polo]{Francisco J. Fern\'{a}ndez-Polo}
\address{Departamento de An{\'a}lisis Matem{\'a}tico, Facultad de
Ciencias, Universidad de Granada, 18071 Granada, Spain.}
\email{pacopolo@ugr.es}

\author[A.M. Peralta]{Antonio M. Peralta}

\address{Departamento de An{\'a}lisis Matem{\'a}tico, Facultad de
Ciencias, Universidad de Granada, 18071 Granada, Spain.}
\email{aperalta@ugr.es}

%\thanks{}



\date{}

\begin{abstract} We prove that every surjective isometry between the unit spheres of two von Neumann algebras admits a unique extension to a surjective real linear isometry between these two algebras.
\end{abstract}

\maketitle
\thispagestyle{empty}

\section{Introduction}

Since 1987, many researchers have been struggling to provide a positive answer to the so-called Tingley's problem. Accordingly to the usual notation, the unit sphere of a Banach $X$ space will be denoted by $S(X)$, while the symbol $\mathcal{B}_X$ will stand for the closed unit ball of $X$. Thirty years ago D. Tingley showed that for any two finite dimensional Banach spaces $X$ and $Y$, every surjective isometry $f:S(X)\to S(Y)$ satisfies $f(-x) = -f(x)$ for every $x\in S(X)$ (see \cite[THEOREM in page 377]{Ting1987}). What is nowadays called \emph{Tingley's problem} asks for a stronger conclusion. It can be stated as follows: Suppose $f: S(X)\to S(Y)$ is a surjective isometry, where $X$ and $Y$ are two arbitrary Banach spaces. Does $f$ admit an extension to a surjective real linear isometry $T: X\to Y$?\smallskip

The surveys \cite{Ding2009} and \cite{YangZhao2014} contain a good exposition on Tingley's problem for some classical Banach spaces. Beside these references, we observe that positive answers to Tingley's problem have been established in the following cases: \begin{enumerate}[$\checkmark$]\item $X$ and $Y$ are $\ell^p (\Gamma)$ spaces with $1\leq p\leq \infty$ (G.G. Ding \cite{Ding2002,Di:p,Di:8} and \cite{Di:1});
\item $X$ and $Y$ are $L^{p}(\Omega, \Sigma, \mu)$ spaces, where $(\Omega, \Sigma, \mu)$ is a $\sigma$-finite measure space and $1\leq p\leq \infty$ (D. Tan \cite{Ta:8, Ta:1} and \cite{Ta:p}); \item $X$ and $Y$ are $C_0(L)$ spaces (R.S. Wang \cite{Wang});
\item $X$ and $Y$ are finite dimensional polyhedral Banach spaces (V. Kadets and M. Mart{\'i}n \cite{KadMar2012});
\item $X$ and $Y$ are finite von Neumann algebras (R. Tanaka \cite{Tan2017b});
\item $X$ and $Y$ are $K(H)$ spaces or compact C$^*$-algebras, where $K(H)$ stands for the space of all compact operators on a complex Hilbert space $H$ (A.M. Peralta and R. Tanaka \cite{PeTan16});
\item $X$ and $Y$ are weakly compact JB$^*$-triples (F.J. Fern{\'a}ndez-Polo and A.M. Peralta \cite{FerPe17});
\item $X$ and $Y$ are spaces of trace class operators on complex Hilbert spaces (F.J. Fern{\'a}ndez-Polo, J.J. Garc{\'e}s, A.M. Peralta and I. Villanueva \cite{FerGarPeVill17});
\item $X$ and $Y$ are $\ell_{\infty}$-sums of $B(H)$ spaces (i.e. atomic von Neumann algebras), where $B(H)$ is the space of all bounded linear operators on a complex Hilbert space, or more generally, atomic JBW$^*$-triples (F.J. Fern{\'a}ndez-Polo and A.M. Peralta \cite{FerPe17b,FerPe17c}).
\end{enumerate}

One of the most interesting questions that remains open in this line is whether Tingley's problem admits a positive answer in the case in which $X$ and $Y$ are general von Neumann algebras. This question has been intriguing recent contributors, and serious difficulties appear because, as we shall see later in detail, the techniques developed in the references dealing with Tingley's problem for C$^*$-algebras (i.e. \cite{Tan2017b,PeTan16,FerPe17,FerPe17b} and \cite{FerPe17c}) are not valid for general von Neumann algebras. This paper is devoted to study this problem. We establish a complete positive answer to Tingley's problem in the case of surjective isometries between the unit spheres of two arbitrary von Neumann algebras (see Theorem \ref{t Tingley von Neumann}). The concrete result proves that if $f : S(M)\to S(N)$ is a surjective isometry between the unit spheres of two von Neumann algebras, then there is a central projection $p$ in $N$ and a Jordan $^*$-isomorphism $J : M\to M$ such that the mapping $T : M\to N,$ $T(x) = f(1) \left( p J (x) + (1-p) J (x)^*\right)$ {\rm(}$x\in M${\rm)} is a surjective real linear isometry and $T|_{S(M)} = f$.\smallskip

As in the study of Tingley's problem for $K(H)$ and $B(H)$ spaces, atomic von Neumann algebras and JBW$^*$-triples, the starting point of our arguments is a geometric property asserting that a surjective isometry between the unit spheres of two Banach spaces $X$ and $Y$ preserves maximal convex sets of the corresponding spheres (\cite[Lemma 5.1$(ii)$]{ChenDong2011}, \cite[Lemma 3.5]{Tan2014}), in order to apply the description of maximal norm closed faces of the closed unit ball of a C$^*$-algebra. The difficulties appear in the fact that, for two arbitrary von Neumann algebras $M$ and $N$, the maximal norm closed faces of $\mathcal{B}_{M}$ and $\mathcal{B}_{N}$ are given by minimal partial isometries in $M^{**}$ and $N^{**}$, respectively. However, a surjective isometry $f:S(A)\to S(B)$ is not, a priori, defined on the unit sphere of $M^{**}$. When $M$ and $N$ coincide with $B(H)$, one of the main results in \cite{FerPe17b} shows that each norm closed face of $\mathcal{B}_{B(H)}$ associated to a non-zero partial isometry in $B(H)$ is mapped by $f$ to a norm closed face of $\mathcal{B}_{B(H)}$ associated to another non-zero partial isometry in $B(H).$ In this paper we improve the results known in this direction by showing that every surjective isometry $f: S(A) \to S(B)$ between the unit spheres of two C$^*$-algebras maps non-zero partial isometries in $A$ to partial isometries in $B$. Furthermore, the norm closed face of $\mathcal{B}_{A}$ associated to a non-zero partial isometry $e$ in $A$ is mapped by $f$ to the norm closed face of $\mathcal{B}_{B}$ associated to $f(e)$ in $B$ (see Theorem \ref{t A}).\smallskip

In the proof of \cite[Theorem 4.12]{Tan2017b} R. Tanaka observed that a surjective isometry between the unit spheres of two finite von Neumann algebras maps unitary elements to unitary elements. After a series of technical results, we establish in Theorem \ref{t Tanaka unitaries general vN} a generalization of this fact by showing that a surjective isometry between the unit spheres of two arbitrary von Neumann algebras preserves unitary elements. This result is later combined with an outstanding result due to O. Hatori and L. Moln{\'a}r asserting that every surjective isometry between the unitary groups of two von Neumann algebras extends to a surjective real linear isometry between the algebras \cite[Corollary 3]{HatMol2014}. Finally, our arguments are culminated with an appropriate application of the theory of \emph{convex combinations of unitary operators in von Neumann algebras} developed by C.L. Olsen and G.K. Pedersen in \cite{OlsPed1986}.

\section{New properties derived from the facial structure}

A common starting point in most of the studies on Tingley's problem is a geometric property asserting that a surjective isometry between the unit spheres of two Banach spaces $X$ and $Y$ preserves maximal convex sets of the corresponding spheres (\cite[Lemma 5.1$(ii)$]{ChenDong2011}, \cite[Lemma 3.5]{Tan2014}). Furthermore, R. Tanaka proves that a surjective as above actually preserves maximal proper norm closed faces of the corresponding closed unit balls of the spaces (compare \cite{Tan2016}).\smallskip

In a recent contribution, J.J. Garc{\'e}s, I. Villanueva and the authors of this note establish a stronger version of the above facts for surjective isometries between the unit spheres of C$^*$-algebras.

\begin{theorem}\label{t all faces von Neumann}\cite{FerGarPeVill17} Let $f: S(A) \to S(B)$ be a surjective isometry between the unit spheres of two C$^*$-algebras. Then the following statements hold:
\begin{enumerate}[$(a)$]\item Let $\mathcal{F}$ be a convex set in $S(A)$. Then $\mathcal{F}$ is a norm closed face of $\mathcal{B}_{A}$ if and only if $f(\mathcal{F})$ is a norm closed face of $\mathcal{B}_{B}$;
\item Given $e\in S(A)$, we have that $e\in \partial_e (\mathcal{B}_A)$ if and only if $f(e)\in \partial_e (\mathcal{B}_B)$.
\end{enumerate}\end{theorem}

\begin{proof} The statement follows from Corollary 2.5 in \cite{FerGarPeVill17} and the comments preceding and following it.
\end{proof}

Motivated by the previous result it worth to spend few paragraphs in reviewing the facial structure of the closed unit ball of a C$^*$-algebra. For this purpose we borrow the next nine paragraphs from the review done in \cite{FerPe17b}.\smallskip

Let $a$ and $b$ be two elements in a C$^*$-algebra $A$. We recall that $a$ and $b$ are orthogonal ($a\perp b$ in short) if $ ab^* = b^* a =0$. Symmetric elements in $A$ are orthogonal if and only if their product is zero. It is known that $$\|a+ b\| =\max\{\|a\|,\|b\| \},$$ for every $a,b\in A$ with $a\perp b$. Given a subset $S\subseteq A$ the symbol $S^{\perp}$ will stand for the set of all elements in $A$ which are orthogonal to every element in $S$, that is, $S^{\perp} :=\{ x\in A : x\perp a, \hbox{ for all } a\in S\}$. \smallskip

For each element $a$ in a C$^*$-algebra $A$, the symbol $|a|$ will denote the element $(a^* a)^{\frac12}\in A$. Throughout this article, for each $x\in A$, $\sigma (x)$ will denote the spectrum of the element $x$. We observe that $\sigma(|a|)\cup \{0\} = \sigma(|a^*|)\cup \{0\}$, for every $a\in A$. Let $a = v |a|$ be the polar decomposition of $a$ in $A^{**}$, where $v$ is a partial isometry in $A^{**}$, which, in general, does not belong to $A$ (compare \cite{S}). It is further known that $v^*v$ is the range projection of $|a|$ ($r(|a|)$ in short), and for each $h\in C(\sigma(|a|)),$ with $h(0)=0$ the element $v h(|a|)\in A$ (see \cite[Lemma 2.1]{AkPed77}).\smallskip

A complete study of the facial structure of the closed unit ball of a C$^*$-algebra was conducted by C.A. Akemann and G.K. Pedersen in \cite{AkPed92}. When $A$ is a von Neumann algebra, weak$^*$ closed faces in $\mathcal{B}_A$ were originally determined by C.M. Edwards and G.T. R\"{u}ttimann in \cite{EdRutt88}, who proved that general weak$^*$ closed faces in $\mathcal{B}_A$ have the form $$F_v=v + (1 - vv^*) \mathcal{B}_{A} (1 - v^*v) = \{x\in \mathcal{B}_A:\ xv^* = vv^*\},$$ for some partial isometry $v$ in $A$. Actually, the mapping $v\mapsto F_v$ is an anti-order isomorphism from the complete lattice of partial isometries in $A$ onto the complete lattice of weak$^*$ closed faces of $\mathcal{B}_A$, where the partial order in the set of partial isometries of $A$ is given by $v\leq u$ if and only if $u = v + (1-vv^*) u (1-v^*v)$ (see \cite[Theorem 4.6]{EdRutt88} or \cite[\S 5]{AkPed92}).\smallskip

However, partial isometries in a general C$^*$-algebra $A$ are not enough to determine all norm closed faces in $\mathcal{B}_A$, even more after recalling the existence of C$^*$-algebras containing no partial isometries. In the general case, certain partial isometries in the second dual $A^{**}$ are required to determine the facial structure of $\mathcal{B}_A$. We recall that a projection $p$ in $A^{**}$ is called \emph{open} if $A\cap (p A^{**} p)$ is weak$^*$ dense in $p A^{**} p$, equivalently there exists an increasing net of positive elements in $A$, all of them bounded by $p$, converging to $p$ in the strong$^*$ topology of $A^{**}$ (see \cite[\S 3.11]{Ped}, \cite[\S III.6 and Corollary III.6.20]{Tak}). A projection $p \in A^{**}$ is said to be \emph{closed} if $1-p$ is open. A closed projection $p$ in $A^{**}$ is \emph{compact} if $p\leq x$ for some norm-one positive element $x \in A$. A partial isometry $v\in A^{**}$ \emph{belongs locally to $A$} if $v^*v$ is a compact projection and there exists a norm-one element $x$ in $A$ satisfying $v = x v^*v$ (compare \cite[Remark 4.7]{AkPed92}). It was shown by C.A. Akemann and G.K. Pedersen that a partial isometry $v$ in $A^{**}$ belongs locally to $A$ if and only if $v^*$ belongs locally to $A$ (see \cite[Lemma 4.8]{AkPed92}).\smallskip

It is shown in \cite[Lemma 4.8 and Remark 4.11]{AkPed92} that ``the partial isometries that belong locally to $A$ are obtained by taking an element $x$ in $A$ with norm 1 and polar decomposition $x = u |x|$ (in $A^{**}$), and then letting $v = ue$ for some compact projection $e$ contained in the spectral projection $\chi_{_{\{1\}}}(|x|)$ of $|x|$ corresponding to the eigenvalue 1.''\smallskip

It should be noted that a partial isometry $v$ in $A^{**}$ belongs locally to $A$ if and only if it is compact in the sense introduced by C.M. Edwards and G.T. R\"{u}ttimann in \cite[Theorem 5.1]{EdRu96}. In this note, we shall mainly use the term compact to refer to those partial isometries in $A^{**}$ belonging locally to $A$. \smallskip

In order to equip the reader with a complete background, we recall a basic tool in describing the facial structure of the closed unit ball $\mathcal{B}_X$ in a complex Banach space $X$. For each $F\subseteq \mathcal{B}_X$  and $G\subseteq \mathcal{B}_{X^*}$, we define
$$ F^{\prime} = \{a \in \mathcal{B}_{X^*} :a(x) = 1\,\, \forall x \in F\},\quad
G_{\prime} = \{x \in \mathcal{B}_X : a(x) = 1\,\, \forall a \in G\}.$$ Then,
$F^{\prime}$ is a weak$^*$ closed face of $\mathcal{B}_{X^*}$ and $G_{\prime}$
is a norm closed face of $\mathcal{B}_X$. Given a convex
set $\mathcal{C}$ we denote by $\partial_{e} (\mathcal{C})$ the set of all extreme points in $\mathcal{C}$.\smallskip

We can now state the result determining the facial structure of the unit ball of a C$^*$-algebra.

\begin{theorem}\label{t faces AkPed}\cite[Theorems 4.10 and 4.11]{AkPed92} Let $A$ be a C$^*$-algebra. The norm closed faces of the unit ball of $A$ have the form $$F_v=\{v\}_{_{''}}=\left(v + (1 - vv^*) \mathcal{B}_{A^{**}} (1 - v^*v)\right)\cap \mathcal{B}_{A} = \{x\in \mathcal{B}_A:\ xv^* = vv^*\},$$ for some partial isometry $v$ in $A^{**}$ belonging locally to $A$ . Actually, the mapping $v\mapsto F_v$ is an anti-order isomorphism from the complete lattice of partial isometries in $A^{**}$ belonging locally to $A$ onto the complete lattice of norm closed faces of $\mathcal{B}_A$. Furthermore, for each weak$^*$ closed face $\mathcal{G}$ of the closed unit ball of $A^{*}$ there exists a unique partial isometry $v$ in $A^{**}$ belonging locally to $A$ such that $\mathcal{G} = \{v\}_{_{'}}$.$\hfill\Box$
\end{theorem}

A non-zero partial isometry $e$ in a C$^*$-algebra $A$ is called minimal if $ee^*$ (equivalently, $e^* e$) is a minimal projection in $A$, that is, $ ee^* A e e^* =\mathbb{C}  ee^*.$ By Kadison's transitivity theorem\label{Kadison transitivity} minimal partial isometries in $A^{**}$ belong locally to $A$, and hence every maximal (norm closed) proper face of the unit ball of a C$^*$-algebra $A$ is of the form \begin{equation}\label{eq max norm-closed faces} \left(v + (1 - vv^*) \mathcal{B}_{A^{**}} (1 - v^*v)\right)\cap \mathcal{B}_{A}
 \end{equation} for a unique minimal partial isometry $v$ in $A^{**}$ (compare \cite[Remark 5.4 and Corollary 5.5]{AkPed92}). As a consequence of this result we can deduce the following: let $\varphi$ be an extreme point of the closed unit ball of the dual space, $A^{*},$ of a C$^*$-algebra $A$. Then there exists a unique minimal partial isometry $w\in A^{**}$ such that \begin{equation}\label{eq support pure atoms and minimal} \{\varphi\} = \{\phi\in S(A^{*}) : \phi (w) = 1\} =\{w\}_{_{'}},
 \end{equation} $$F_w = \{x\in \mathcal{B}_A:\ xw^* = ww^*\} =\{x\in \mathcal{B}_A:\ \varphi(x) = 1\} $$ $$= \left(w + (1 - ww^*) \mathcal{B}_{A^{**}} (1 - w^*w)\right)\cap \mathcal{B}_{A}.$$

Combining the results reviewed in the previous paragraph we get the following.

\begin{theorem}\label{t first correspondence between faces and compact partial isometries in the bidual for a surjective isometry} Let $f: S(A) \to S(B)$ be a surjective isometry between the unit spheres of two C$^*$-algebras. Then the following statements hold:
\begin{enumerate}[$(a)$]\item For each non-zero compact partial isometry $e\in A^{**}$ there exists a unique (non-zero) compact partial isometry $\phi_{f} (e)\in B^{**}$ such that $f (F_e) = F_{\phi_f(e)};$
\item The mapping $e\mapsto \phi_f(e)$ defines an order preserving bijection between the sets of non-zero compact partial isometries in $A^{**}$ and the set of non-zero compact partial isometries in $B^{**}$;
\item $\phi_{f}$ maps minimal partial isometries in $A^{**}$ to minimal partial isometries in $B^{**}$.
$\hfill\Box$
\end{enumerate}
\end{theorem}

The handicap of the statement of the above theorem is that given a non-zero partial isometry $e$ in $A$ (which is clearly compact in $A^{**}$), we cannot guarantee that $\phi_f (e)$ belongs to $B$, we neither know that $f(e)$ is partial isometry in $B$, nor that, even assuming a positive answer to the above statements, $\phi_f (e)$ coincides with $f(e)$. We shall fill all these gaps in this section.\smallskip

We recall that every C$^*$-algebra is a JB$^*$-triple in the sense introduced in \cite{Ka83} with respect to the triple product given by \begin{equation}\label{eq C*-triple prod} \{a,b,c\} = 2^{-1} (a b^* c + c b^* a)
 \end{equation}(see \cite[pages 522-523]{Ka83}).\smallskip

We shall also consider the Peirce decomposition associated with a partial isometry $e$ in a C$^*$-algebra $A$. It is easy to check that $A= A_0(e) \oplus A_1 (e)\oplus A_2(e)$, where $A_0(e)$, $A_1(e)$ and $A_2 (e)$ are the Peirce subspaces given by $$A_2(e) = ee^* A e^*e=\{x\in A : \{e,e,x\} = x\} ,$$ $$A_1(e) = (1-ee^*) A e^*e \oplus ee^* A(1-e^*e)=\{x\in A : \{e,e,x\} = 1/2 x\},$$ and $A_0(e) = (1-ee^*) A (1-e^*e)=\{x\in A : \{e,e,x\} = 0\}$. The Peirce subspaces need not be C$^*$-subalgebras of $A$. However, it is easy to see that $A_2(e)$, $A_1(e),$ and $A_0(e)$ are all closed under the triple product defined in \eqref{eq C*-triple prod}, and consequently they are all JB$^*$-subtriples of $A$. Furthermore, the Peirce-2 subspace $A_2(e)$ is a JB$^*$-algebra with product $a\circ_e b:= \{a,e,b\}$ and involution $a^{\sharp_e}:=\{e,a,e\}$, respectively. The natural projection of $A$ onto $A_j(e)$ will be denoted by $P_j(e)$. It is known that $P_2(e) (x) = ee^* x e^*e$, $P_1(e) (x) = ee^* x (1-e^*e) + (1-ee^*) x e^*e$, and $P_0(e) (x) = (1-ee^*) x (1-e^*e)$ for all $x\in A$. It is easy to check that, for every partial isometry $e\in A$, we have $A_2 (e)\perp A_0(e)$. \smallskip

The next result has been taken from \cite{FerPe17}.

\begin{proposition}\label{p l 3.4 new}{\rm\cite[Proposition 2.2]{FerPe17}} Let $v$ and $x$ be norm-one elements in a C$^*$-algebra $A$. Suppose that $v$ is a minimal partial isometry and $\|v-x\| = 2$. Then $x= -v + P_0(v) (x)$. $\hfill\Box$
\end{proposition}

It seems natural to ask what happens if in the above result $v$ is not a minimal partial isometry.

\begin{proposition}\label{p l 3.4 new new} Let $v$ and $x$ be norm-one elements in a C$^*$-algebra $A$. Suppose that $v$ is a partial isometry and $\|v-x\| = 2$. Then there exists a minimal partial isometry $w\in A^{**}$ such that $w\leq v$ and $x= -w + P_0(w) (x)$. In particular the element $P_2 (v) (v-x)$ has norm $2$.
\end{proposition}

\begin{proof} By the Hahn-Banach theorem the set $$\mathcal{C} = \Big\{\frac{v-x}{2}\Big\}' :=\{\phi\in A^* : \|\phi\|=1, \ \phi (v-x) =2 \}$$ is a non-empty weak$^*$ closed face of $\mathcal{B}_{A^{*}}$. By the Krein-Milman theorem,  there exists $\varphi\in \partial_{e}(\mathcal{B}_{A^{*}})$ such that $\varphi\in \mathcal{C}$. By \eqref{eq support pure atoms and minimal} there exists a minimal partial isometry $w$ in $A^{**}$ such that $$F_w = \{x\in \mathcal{B}_A:\ xw^* = ww^*\} =\{z\in \mathcal{B}_A:\ \varphi(z) = 1\}=\{\varphi\}_{_{'}}.$$ By construction $\mathcal{C}_{_{'}} := \{z\in S(A) : \phi (z) =1\ \forall \phi\in \mathcal{C} \}$ is a non-empty norm closed face of $\mathcal{B}_{A}$ which is trivially contained in $F_w$ (just observe that $\varphi\in \mathcal{C}$).\smallskip

Since $\varphi (v-x) =2,$ we get $\varphi (v) =1 = \varphi (x),$ and thus $v\in \{\varphi\}_{_{'}} = F_{w},$ which proves that $v\geq w$. On the other hand, $2 = \varphi (w-x)\leq \| w-x\| \leq2$. Therefore, by applying Proposition \ref{p l 3.4 new} in $A^{**}$, we derive that
$x= -w + P_0(w) (x)$.\smallskip

Finally, since $$v-x \!= 2 w + (v-w) - P_0(w) (x)=2 w +P_0(w) ( (v-w) - x)= 2 w +P_0(w) ( v - x),$$ it follows that $P_2(v) (v-x) = 2 w + P_2(v) P_0(w)(v-x) = 2 w + P_0(w) P_2(v)(v-x).$ Now, we deduce from the orthogonality of $w$ and $P_0(w) P_2(v)(v-x)$ that $$\|P_2(v) (v-x) \| =\max\{ 2 \|w\| ,\|P_0(w) P_2(v)(v-x)\| \} =2.$$
\end{proof}

We recall now some basic notions on ultraproducts. Given an ultrafilter $\mathcal{U}$ on an index set $I$, and a family $(X_i)_{i\in I}$ of Banach spaces, the symbol $(X_i)_{\mathcal{U}}$ will denote the \emph{ultraproduct} of the $X_i,$ and if $X_i=X$ for all $i$, we write $(X)_{\mathcal{U}}$ for the ultrapower of $X$. As usually, elements in the Banach space $(X_i)_{\mathcal{U}}$ will be denoted in the form $\widetilde{x}=[x_i]_{\mathcal{U}}$, where $(x_i)$ is called a representing family or a representative of $\widetilde{x}$, and $\left\|{\widetilde{x}}\right\|=\lim_{\mathcal{U}}\left\|{x_i}\right\|$ independently of the representative. The basic facts and definitions concerning ultraproducts can be found in \cite{Hein80}.\smallskip

C$^*$-algebras are stable under $\ell_{\infty}$-sums (see \cite[Proposition 3.1$(ii)$]{Hein80}), that is, the ultraproduct $\ell_{\infty} (A_i)$ of a family $(A_i)$ of C$^*$-algebras is a C$^*$-algebra with respect to the natural operations.\smallskip

Our next goal is a quantitative version of the previous Proposition \ref{p l 3.4 new new}.

\begin{proposition}\label{p l 3.4 new new quatitative} For each $1>\varepsilon >0$ there exists $\delta>0$ satisfying the following property: given a C$^*$-algebra $A$, a partial isometry $e$ in $A$ and $x\in S(A)$ with $\|e-x \| > 2 -\delta$ we have $\| P_2 (e) (e -x) \|> 2-\varepsilon.$
\end{proposition}

\begin{proof} Arguing by contradiction, we assume the existence of $\varepsilon>0$, such that for each natural $n$ we can find a C$^*$-algebra $A_n,$ a non-zero partial isometry $e_n\in A_n,$ and $x_n$ in $S(A)$ with $\| e_n-x_n \| > 2 -\frac{1}{n}$ and $\| P_2 (e_n) (e_n -x_n) \|\leq  2-\varepsilon.$\smallskip

Let us consider a free ultrafilter $\mathcal{U}$ over $\mathbb{N}$, and let $(A_n)_{\mathcal{U}}$ denote the ultraproduct of the family $(A_n)$. Clearly the element $[e_n]_{\mathcal{U}}$ is a partial isometry in $(A_n)_{\mathcal{U}}$.\smallskip

Since $\|[x_n]_{\mathcal{U}}\| = 1$, and $2\geq \| [e_n]_{\mathcal{U}} - [x_n]_{\mathcal{U}}\| = \lim_{\mathcal{U}} \| e_n-x_n\| \geq \lim_{\mathcal{U}}  2-\frac1n =2,$ by Proposition \ref{p l 3.4 new new} we can find a minimal partial isometry $\widetilde{w}$ in $((A_n)_{\mathcal{U}})^{**}$ such that
$$[x_n]_{\mathcal{U}} = -\widetilde{w} + P_0 (\widetilde{w}) [x_n]_{\mathcal{U}},\hbox{ and } [e_n]_{\mathcal{U}} = \widetilde{w} + P_0 (\widetilde{w}) [e_n]_{\mathcal{U}}.$$ Therefore, $$ 2-\varepsilon \geq \left\| [P_2 (e_n) (e_n-x_n)]_{\mathcal{U}} \right\| = \left\| P_2 ([e_n]_{\mathcal{U}}) \left([e_n]_{\mathcal{U}} -  [x_n]_{\mathcal{U}}\right) \right\| $$ $$ = \left\| P_2 ([e_n]_{\mathcal{U}}) \Big( 2 \widetilde{w} + P_0 (\widetilde{w})  \left([e_n]_{\mathcal{U}} -  [x_n]_{\mathcal{U}}\right) \Big) \right\| $$ $$ = \left\|  2 \widetilde{w} + P_2 ([e_n]_{\mathcal{U}})\Big(P_0 (\widetilde{w})  \left([e_n]_{\mathcal{U}} -  [x_n]_{\mathcal{U}}\right) \Big) \right\| $$ $$ = \left\|  2 \widetilde{w} + P_0 (\widetilde{w})  \Big( P_2 ([e_n]_{\mathcal{U}}) \left([e_n]_{\mathcal{U}} -  [x_n]_{\mathcal{U}}\right) \Big) \right\|=2,$$ which is impossible.
\end{proof}

We continue our study with a strengthened version of \cite[Theorem 2.3]{FerPe17b} for non-necessarily minimal partial isometries.

\begin{theorem}\label{t surjective isometries map partial isometries into points of strong subdiff}
Let $A$ and $B$ be C$^*$-algebras, and suppose that $f: S(A) \to S(B)$ is a surjective isometry. Let $e$ be a non-zero partial isometry in $A$. Then $1$ is isolated in the spectrum of $|f(e)|$.
\end{theorem}

\begin{proof} The property established in Proposition \ref{p l 3.4 new new quatitative} allows us to adapt a refinement of the original proof in \cite[Theorem 2.5]{FerPe17c}.\smallskip

Arguing by contradiction, we assume that 1 is not isolated in $\sigma(|f(e)|)$. Let $f(e) = v |f(e)|$ denote the polar decomposition of $f(e)$ in $B$ with $v$ a partial isometry in $B^{**}$. It is well known that the subalgebra $B_{|f(e)|}$ of $B$ generated by $|f(e)|$ identifies with $C_0 (\sigma (|f(e)|))$.\smallskip

By Theorem \ref{t first correspondence between faces and compact partial isometries in the bidual for a surjective isometry} there exists a non-zero compact partial isometry $u\in B^{**}$ such that \begin{equation}\label{eq faces thm 1} f(F_e) = f((e + A_0(e)) \cap \mathcal{B}_A)    = F_u= (u + B_0^{**}(e)) \cap \mathcal{B}_B.
\end{equation}

For each natural $n,$ we define $\hat{a}_n, \hat{b}_n$ the elements in $B_{|f(e)|}$ given by:
$$\hat{a}_n(t):=\left\{%
\begin{array}{ll}
    \frac{n t}{n-1}, & \hbox{if $0\leq t\leq 1-\frac1n$} \\
    \hbox{affine}, & \hbox{if $1-\frac1n\leq t\leq 1-\frac{1}{2n}$} \\
    0, & \hbox{if $1-\frac{1}{2n}\leq t\leq 1$} \\
\end{array}%
\right. \  ; \hat{b}_n(t):=\left\{%
\begin{array}{ll}
    0, & \hbox{if $0\leq t\leq 1-\frac{1}{2n}$} \\
    \hbox{affine}, & \hbox{if $1-\frac{1}{2n}\leq t\leq 1$} \\
    1, & \hbox{if $ t=1$.} \\
\end{array}%
\right.$$

Clearly $\hat{a}_n,\hat{b}_n\in S(B)$ and $\hat{a}_n\perp\hat{b}_n$. If we set $\hat{x}_n = v \hat{a}_n$ and $\hat{y}_n = v \hat{b}_n$, we get two elements in $S(B)$ (compare \cite[Lemma 2.1]{AkPed77}) satisfying $\hat{x}_n\perp\hat{y}_n$.\smallskip

We shall prove that $\hat{y}_n\in (u+B^{**}_0 (u))\cap \mathcal{B}_B = F_{u}$. Namely, by \eqref{eq faces thm 1} we know that $f(e) = u + P_0(u) (f(e))$ in $B^{**}$, and thus $|f(e)| = u^*u  + |P_0(u) (f(e))|$ in $B^{**},$ which shows that $u^*u \leq \chi_{_{\{1\}}}$ the characteristic function of the set $\{1\}$ in $C_0 (\sigma (|f(e)|))^{**}$. Therefore $\hat{y}_n = u + P_0(u) (\hat{y}_n) \in F_u$.\smallskip

The elements $x_n=f^{-1}(-\hat{x}_n)\in S(A),$ and $y_n=f^{-1}(\hat{y}_n)\in F_e=(e+A^{**}_0(e))\cap \mathcal{B}_A =(e+A_0(e))\cap \mathcal{B}_A$ (equivalently, $y_n = e + P_0 (e) (y_n)$) satisfy $$1=\|\hat{x}_n+\hat{y}_n\|=\|\hat{y}_n- (-\hat{x}_n)\|=\|y_n-x_n\|,$$  $$2-\frac1n=\|f(e)+\hat{x}_n\|=\|f(e)-(-\hat{x}_n)\|=\|e-x_n\|.$$

Applying Proposition \ref{p l 3.4 new new quatitative} we can find a natural $n_0$ such that $$\|P_2 (e) (e- x_{n_0})\|>\frac32 >1.$$

Finally, the inequalities $$1\geq \|P_2(e) (y_{n_0} -x_{n_0} )\|=\|P_2(e)(e+P_0(e) (y_{n_0})-x_{n_0})\|=\|P_2(e)(e-x_{n_0})\|>\frac32,$$ give the desired contradiction.
\end{proof}

We can solve now all concerns appearing after Theorem \ref{t first correspondence between faces and compact partial isometries in the bidual for a surjective isometry} via a generalization of \cite[Theorem 2.5]{FerPe17b}.

\begin{theorem}\label{t A}
Let $f: S(A) \to S(B)$ be a surjective isometry between the unit spheres of two C$^*$-algebras. Then $f$ maps non-zero partial isometries in $A$ into non-zero partial isometries in $B$. Moreover, for each non-zero partial isometry $e$ in $A$, $\phi_f (e) = f(e)$ and  there exits a surjective real linear isometry $$T_{e} : (1 - ee^*) A  (1 - e^*e)\to  {(1 - f(e)f(e)^*) B (1 - f(e)^*f(e))}$$ such that $$ f(e + x) = f(e) + T_e (x), \hbox{ for all $x$ in } \mathcal{B}_{(1 - ee^*) A  (1 - e^*e)}.$$ In particular the restriction of $f$ to the face $F_{e} =e + (1 - ee^*) \mathcal{B}_A (1 - e^*e)$ is a real affine function.
\end{theorem}

\begin{proof} Let us pick a non-zero partial isometry $e$ in $A$, and let $f(e) = v |f(e)|$ be the polar decomposition of $f(e)$ in $B^{**}$. By Theorem \ref{t surjective isometries map partial isometries into points of strong subdiff}, $1$ is an isolated point in $\sigma (|f(e)|)$.\smallskip

As in the proof of the previous theorem, by Theorem \ref{t first correspondence between faces and compact partial isometries in the bidual for a surjective isometry} we can find a (unique) non-zero compact partial isometry $u=\phi_f (e)\in B^{**}$ such that \begin{equation}\label{eq faces thm 2} f(F_e) = f((e + A_0(e)) \cap \mathcal{B}_A)    = F_u= (u + B_0^{**}(e)) \cap \mathcal{B}_B.
\end{equation}

We shall first prove that $u\in B$. As before, we shall identify the C$^*$-subalgebra $B_{|f(e)|}$ generated by $|f(e)|$ with $C_0 (\sigma (|f(e)|))$. Since $1$ is isolated in $\sigma (|f(e)|)$, the element ${w} = v \chi_{\{1\}} (|f(e)|)$ is a partial isometry in $B$ (cf. \cite[Lemma 2.1]{AkPed77}). Having in mind that $f(e) \in F_u$, we have $f(e) = u + P_0(u) (f(e))$. Since ${w} = v \chi_{\{1\}} (|f(e)|)$, $|f(e)| = |u| + |P_0(u) (f(e))|$ with $|u|\perp |P_0(u) (f(e))|$, and hence $$u \leq {w} = v \chi_{\{1\}} (|u|) + v \chi_{\{1\}} (|P_0(u) (f(e))|).$$

If the partial isometry ${w}-u$ is non-zero, the projection $({w}-u)^* ({w}-u) = {w}^*{w} - u^* u $ lies in $B^{**}$ and it is bounded by ${w}^* {w},$ which is a projection in $B$. The (unital) hereditary C$^*$-subalgebra $ {w}^* {w} B {w}^* {w}$ is weak$^*$ dense in ${w}^* {w} B^{**} {w}^* {w}$ (just observe that the product of $B^{**}$ is separately weak$^*$ continuous by Sakai's theorem \cite[Theorem 1.7.8]{S}). Since $u^*u$ is compact in $B^{**}$, the projection $1-u^* u$ is open in $B^{**}$, and hence ${w}^*{w} - u^* u = {w}^*{w} (1-u^*u) {w}^*{w}$ is an open projection in $ {w}^* {w} B^{**} {w}^* {w}$. We can therefore find a positive norm-one element ${a}\leq {w}^*{w} - u^* u $ in  ${w}^* {w} B {w}^* {w}$. The element ${x} = {w} {a}$ lies in $B$, has norm one, and ${x} \perp u$ because $u {x}^* = u {a} {w}^* =   u ({w}^*{w} - u^* u) {a} {w}^*=0$, and ${x}^* u = {a} {w}^* u = {a}  {u}^* u = {a} ({w}^*{w} - u^* u) {u}^* u =0.$\smallskip

In these circumstances, by Proposition \ref{p l 3.4 new new} we have $$2=\|f(e)+{x}\|=\|f(e)-(-{x})\| =\| e - f^{-1}(-{x})\| = \| P_2 (e) (e - f^{-1}(-{x})) \|.$$ Let us observe that, since $u\leq w\leq v$, we can write $$f(e) -x = u + P_0(u) (f(e)) - w  a = u + (1-uu^*) v |f(e)| (1-u^* u) - w (w^*w -u^*u) a $$ $$= u + v (1-u^*u) |f(e)| (1-u^* u) - v (w^*w -u^*u) a (w^*w -u^*u)$$ $$= u + v (v^*v-u^*u) |f(e)| (v^*v-u^* u) - v (v^*v -u^*u) a (v^*v -u^*u)$$ $$= u + v (v^*v-u^*u) (|f(e)|-a) (v^*v-u^* u).$$ Since $|f(e)|$ and $a$ are positive elements in the closed unit ball of the C$^*$-algebra $(w^*w -u^*u) B^{**}(w^*w -u^*u)$, we have $-1 \leq -a \leq |f(e)|-a\leq |f(e)|\leq 1$, and consequently $\| |f(e)|-a\|\leq 1$. Now, the orthogonality of the elements $u$ and $v (v^*v-u^*u) (|f(e)|-a) (v^*v-u^* u)$ implies that $\|f(e)-x\| =1$. It follows that $$ f(e) -x = u + v (v^*v-u^*u) (|f(e)|-a) (v^*v-u^* u)\in F_u.$$

We have show in the previous paragraph that $f(e)-{x} \in F_u$. Denoting by $z= f^{-1} (f(e)-{x} )$, we get $z\in F_e$, and $$1=\|f(e)\|=\|f(e)-{x}+{x}\|=\|(f(e)-{x})-(-{x})\|=\|z- f^{-1}(-{x})\|$$ $$=\|e+ P_0(e)(z)- f^{-1}(-{x})\|\geq  \|P_2 (e) (e+ P_0(e)(z)- f^{-1}(-{x}))\| $$ $$=\|P_2 (e) (e - f^{-1}(-{x}))\|=2,$$ which is impossible.\smallskip

The previous arguments show that $\phi_{f}(e)\in B$ for every non-zero partial isometry $e\in A$. A standard argument based on previous contributions (cf. \cite{PeTan16} and \cite{FerPe17b,FerPe17c}) and Mankiewicz's theorem (see \cite{Mank1972}) can be now applied to conclude the proof. We sketch an argument. Pick a non-zero partial isometry $e\in A$ with $u=\phi_{f}(e)\in B$. Since $$f\left( e + \mathcal{B}_{A_0(e)} \right)=f((e + A_0(e)) \cap \mathcal{B}_A ) = F_u = (u + B_0(u)) \cap \mathcal{B}_B = u + \mathcal{B}_{B_0(u)},$$ denoting by $\mathcal{T}_{x_0}$ the translation with respect to $x_0$ (i.e. $\mathcal{T}_{x_0} (x) = x+x_0$), the mapping $f_{e} = \mathcal{T}_{u}^{-1}|_{F_u} \circ f|_{F_e} \circ \mathcal{T}_{e}|_{\mathcal{B}_{A_0(e)}}$ is a surjective isometry from $\mathcal{B}_{A_{0} (e)}$ onto $\mathcal{B}_{B_0(u)}$. Mankiewicz's theorem \cite{Mank1972} implies the existence of a surjective real linear isometry $T_{e} : A_0(e) \to  B_0(u)$ such that $f_{e} = T_{e}|_{\mathcal{B}_{A_0(e)}}$ and hence $$ f(e + x) = u + T_e (x), \hbox{ for all $x$ in } \mathcal{B}_{A_0(e)}.$$ In particular $f(e) = u=\phi_{f}(e)$. Now, since $$f|_{F_e}= \mathcal{T}_{u}|_{\mathcal{B}_{B_0(u)}} \circ f_{e} \circ \mathcal{T}_{e}^{-1}|_{F_e}= \mathcal{T}_{u}|_{\mathcal{B}_{B_0(u)}} \circ T_{e} \circ \mathcal{T}_{e}^{-1}|_{F_e},$$ we deduce that $f|_{F_e}$ is a real affine function.
\end{proof}

We return to the mapping given by Theorem \ref{t first correspondence between faces and compact partial isometries in the bidual for a surjective isometry} to explore some additional properties. Our next result, which asserts that the mapping $\phi_f$ preserves antipodal points for minimal partial isometries in $A^{**}$, is a crucial step in our arguments.

\begin{theorem}\label{t B} Let $f: S(A) \to S(B)$ be a surjective isometry between the unit spheres of two C$^*$-algebras, and let $\phi_f$ be the mapping given by Theorem \ref{t first correspondence between faces and compact partial isometries in the bidual for a surjective isometry}. Then, for each minimal partial isometry $v$ in $A^{**}$ we have $\phi_f (-v) = - \phi_f (v)$.
\end{theorem}

\begin{proof} By Kadison's transitivity theorem every minimal partial isometry $e\in A^{**}$ is compact (cf. comments in page \pageref{Kadison transitivity}). Theorem \ref{t first correspondence between faces and compact partial isometries in the bidual for a surjective isometry}, $\phi_f (e)$ and $\phi_f (-e)$ are minimal partial isometries in $B^{**}$. The same arguments in page \pageref{Kadison transitivity} show the existence of two functionals $\varphi_0$ and $\varphi_1$ in $\partial_{e}(\mathcal{B}_{B^*})$, and a norm-one element $y\in B$ such that $$y= \phi_f (-e) + P_0(\phi_f(-e)) (y), $$ $$\{\phi_f (e)\}_{_{'}} =\{\varphi_0\}, \hbox{ and } \{\phi_f (-e)\}_{_{'}} =\{\varphi_1\}.$$

\begin{equation}\label{C1}\hbox{We claim that for each norm-one element $x\in B$ such that}\hspace{2.4cm}
\end{equation} $$x= \phi_f (e) + P_0(\phi_f(e)) (x) \hbox{ we have  }\varphi_1 (x) = -1.$$

Otherwise, there exists $x$ as above satisfying $|\varphi_1 (x)+1| =\varepsilon>0$. Since $x\in B$ the set $\mathcal{O} = \{ \psi\in B^* : |\psi (x)+1| > \frac{\varepsilon}{2} \}$ is a weak$^*$ open neighborhood of $\varphi_1$ in $B^*$.\smallskip

Since $\phi_f (-e) \phi_f (-e)^*$ is a compact projection in $B^{**}$, we can find a decreasing net $(z_{\lambda})_{\Lambda}$ of positive elements in $B$ converging to $\phi_f (-e) \phi_f (-e)^*$ in the weak$^*$ topology of $B^{**}$, with $\|z_{\lambda}\|=1$ and $\phi_f (-e) \phi_f (-e)^*\leq z_{\lambda},$ for every $\lambda$. Since the product of $B^{**}$ is separately weak$^*$ continuous, the net $(y_{\lambda})=(z_{\lambda} y)$ converges to $\phi_f (-e) \phi_f (-e)^* y = \phi_f (-e)$ in the weak$^*$ topology of $B^{**}$.\smallskip

For each $\lambda \in \Lambda$ we set $C_{\lambda} :=\{y_{\lambda}\}'= \{\psi\in B^* : \psi (y_{\lambda})=1 =\|\psi\| \}$, $C_{0} = \{y\}'= \{\psi\in B^* : \psi (y)=1 =\|\psi\| \}$. Clearly $C_0$ and $C_{\lambda}$ are non-empty weak$^*$ closed faces of $\mathcal{B}_{B^{*}}$. Let us fix $\lambda_1\leq \lambda_2$ in $\Lambda$, and $\psi\in C_{\lambda_2}$. Since $1 = \psi (y_{\lambda_2}) = \psi (z_{\lambda_2} y)$ and $z_{\lambda_2}$ is a norm-one positive element in $B$, the functional $\psi ( \cdot y)$ is a state in $B^{*}$. Therefore $1 = \psi (y_{\lambda_2}) = \psi (z_{\lambda_2} y) \leq \psi (z_{\lambda_1} y)\leq \psi (1 y)\leq 1,$ and then $$ C_{\lambda_2} \subseteq C_{\lambda_1} \subseteq C_0.$$ Furthermore, since $(y_{\lambda})\to \phi_f (-e)$ in the weak$^*$ topology, it is not hard to check that $ \{\varphi_1\} = \{\phi_f (-e)\}_{_{'}} =\bigcap_{\lambda} C_{\lambda} .$ Therefore, $\{C_{\lambda}: \lambda\in \Lambda\}$ is a decreasing family of non-empty weak$^*$ compact sets with non-empty finite intersections, whose intersection is  $\{\varphi_1\}\subset \mathcal{O},$ and the latter is weak$^*$ open. A standard topological argument proves the existence of $\lambda_0\in \Lambda$ such that $C_{\lambda_0}=\{y_{\lambda_0}\}'\subset\mathcal{O}$.\smallskip

Now, applying that $x= \phi_f (e) + P_0(\phi_f(e)) (x) \in \{\phi_f(e)\}_{_{''}}=F_{\phi_f(e)} = f(F_e)$, where $F_e$ and $F_{\phi_f (e)}$ denote the faces of $\mathcal{B}_{A}$ and $\mathcal{B}_{B}$ defined by $e$ and $\phi_f(e)$, respectively (see Theorem \ref{t first correspondence between faces and compact partial isometries in the bidual for a surjective isometry}), we can find $a\in F_e$ such that  $f(a) =x$.\smallskip

On the other hand, $y_{\lambda_0} = z_{\lambda_0} y =\phi_f (-e) + P_0(\phi_f(-e)) (y_{\lambda_0})\in \{\phi_f(-e)\}_{_{''}}=F_{\phi_f(-e)} = f(F_{-e}).$ Let us pick $b\in F_{-e}$ satisfying $f(b) = y_{\lambda_0}$. By the hypothesis on $f$ we have $$\| x-y_{\lambda_0}\| = \|a-b\| = \| e + P_0(e) (a) - (-e + P_0(-e) (b))\| = \|2 e + P_0(e) (a-b)\|=2,$$ which assures that $\displaystyle \Big\|\frac{x-y_{\lambda_0}}{2}\Big\| =1$. An application of the Hahn-Banach theorem combined with Krein-Milman's theorem proves the existence of a functional $\psi\in \partial_{e}(\mathcal{B}_{B^*})$ satisfying $\psi \Big(\frac{-x+y_{\lambda_0}}{2}\Big) =1$. Therefore $\psi( y_{\lambda_0}) =1=- \psi( x),$ and hence $\psi\in C_{\lambda_0}=\{y_{\lambda_0}\}'\subset\mathcal{O}$, and $0=|\psi (x)+1| =\varepsilon>0$, which is impossible. This finishes the proof of \eqref{C1}.\smallskip

By the arguments in page \pageref{Kadison transitivity} we can find norm-one element $x\in B$ such that $x= \phi_f (e) + P_0(\phi_f(e)) (x)$. Since $\phi_f (e) \phi_f (e)^*$ is a compact projection in $B^{**}$, we can find a decreasing net $(h_{\lambda})_{\Lambda}$ of positive elements in $B$ converging to $\phi_f (e) \phi_f (e)^*$ in the weak$^*$ topology of $B^{**}$, with $\|h_{\lambda}\|=1$ and $\phi_f (e) \phi_f (e)^*\leq h_{\lambda},$ for every $\lambda$. By the separate weak$^*$ continuity of the product in $B^{**}$, the net $(x_{\lambda})=(h_{\lambda} x)$ converges to $\phi_f (e) \phi_f (e)^* x = \phi_f (e)$ in the weak$^*$ topology of $B^{**}$. Each $x_\lambda$ has norm-one and satisfies $x_{\lambda}=h_{\lambda} x = \phi_f (e) + P_0(\phi_f(e)) (x_{\lambda}).$ By applying \eqref{C1} to each $x_{\lambda}$ we deduce that $\varphi_1 (x_{\lambda}) = -1,$ for every $\lambda$. Since $(x_{\lambda})\to \phi_f (e)$ in the weak$^*$ topology of $B^{**}$, it follows that $\varphi_1 (\phi_f (e)) =-1$.\smallskip

Finally, by the minimality of $\phi_f (-e)$ in $B^{**}$ it is known that $P_2 (\phi_f (-e)) (x) = \varphi_1 (x) \phi_f (-e)$, for every $x\in B^{**}$, and thus $P_2 (\phi_f (-e)) (\phi_f (e)) = - \phi_f (-e)$. By minimality $\phi_f (-e) = -\phi_f (e)$ (compare, for example, \cite[Lemma 3.1]{FerPe18} or \cite[Lemma 1.6 and Corollary 1.7]{FriRu85}).
\end{proof}

We continue with a technical proposition derived from the results about the facial structure of the closed unit ball of a C$^*$-algebra.

\begin{proposition}\label{p technical} Let $e$ and $v$ be compact partial isometries in the bidual, $A^{**}$, of a C$^*$-algebra $A$. \begin{enumerate}[$(a)$] \item If for each $\varphi\in \partial_{e} (\{e\}_{_{'}})$ we have $\varphi (v) =1$, then $v = e + P_0 (e) (v)$;
\item If for each minimal partial isometry $w\in B^{**}$ with $e = w + P_0(w) (e)$ we have $v = w + P_0(w) (v)$, then $v = e + P_0 (e) (v)$.
\end{enumerate}
\end{proposition}

\begin{proof} Since $e$ is compact the set $\{e\}_{_{'}}$ is a non-empty weak$^*$ closed, and hence weak$^*$ compact, face of $\mathcal{B}_{A^{*}}$ (see Theorem \ref{t faces AkPed}). By the Krein-Milman theorem $\{e\}_{_{'}} = \overline{co}^{w^*} \partial_{e} (\{e\}_{_{'}}).$ By hypothesis, for each norm-one element $y\in A$ with $y = v+P_0(v) (y)$ and each $\varphi\in \partial_{e} (\{e\}_{_{'}})$ we have $1=  \varphi (v)$, and then $\varphi (x) = \varphi (vv^* x)$ for all $x\in A^{**}$, which implies that $\varphi (y) = \varphi (vv^* y) = \varphi (v) =1.$ By considering the weak$^*$ closed convex hull of $\partial_{e} (\{e\}_{_{'}})$, and having in mind that $y\in A$, we get $\varphi (y)=1,$ for every $y$ as above and every $\varphi \in \{e\}_{_{'}}$.\smallskip

Since $v$ is compact, arguing as in the proof of Theorem \ref{t B} we can find a net $(y_{\lambda})$ of norm-one elements in $A$ such that $y_{\lambda} = v +P_0(v) (y_{\lambda})$ for every $\lambda$ and $(y_{\lambda})\to v$ in the weak$^*$ topology of $A^{**}$. By the previous paragraph, for each $\lambda,$ we have $\varphi(y_{\lambda})=1$ for all $\varphi\in \{e\}_{_{'}}$. Taking weak$^*$ limit we prove $\varphi(v)=1$ for all $\varphi\in \{e\}_{_{'}}$, equivalently, $\{e\}_{_{'}}\subseteq \{v\}_{_{'}} \Leftrightarrow \{v\}_{_{''}} = F_v\subseteq \{e\}_{_{''}}=F_e \Leftrightarrow v = e + P_0 (e) (v)\Leftrightarrow  e \leq v.$\smallskip

The final statement is clear because, as we commented before, the element in $ \partial_{e} (\{e\}_{_{'}})$ are in one-to-one correspondence with those minimal partial isometries $w\in B^{**}$ with $e = w + P_0(w) (e)$.
\end{proof}

As a consequence of the above result we can derive now a result in the line started by Tingley.

\begin{theorem}\label{t C} Let $A$ and $B$ be C$^*$-algebras, and let $f: S(A) \to S(B)$ be a surjective isometry. Then, for each non-zero compact partial isometry $e$ in $A^{**}$ we have $\phi_f (-e) = -\phi_f (e)$, where $\phi_f$ is the mapping given by Theorem \ref{t first correspondence between faces and compact partial isometries in the bidual for a surjective isometry}. Consequently, for each nono-zero partial isometry $e\in A$ we have $f(-e) = -f(e)$.
\end{theorem}

\begin{proof} Let $e$ be a non-zero compact partial isometry in $A^{**}$. The elements $\phi_f (e)$ and $\phi_f(-e)$ are compact partial isometries in $B^{**}$. Let us pick an arbitrary minimal partial isometry $\hat{w}$ in $B^{**}$ such that $\phi_f(-e) = \hat{w} + P_0(\hat{w}) (\phi_f(-e))$. Theorem \ref{t first correspondence between faces and compact partial isometries in the bidual for a surjective isometry}$(c)$ assures the existence of a minimal partial isometry $w$ in $A^{**}$ with $\phi_f (w) =\hat{w}$. A new application of Theorem \ref{t first correspondence between faces and compact partial isometries in the bidual for a surjective isometry}$(b)$ tells that  $-e = {w} + P_0({w}) (-e)$ (equivalently, $e = {-w} + P_0({-w}) (e)$), and Theorem \ref{t B} gives $-\phi_f(w)=\phi_f(-w),$ and thus by Theorem \ref{t first correspondence between faces and compact partial isometries in the bidual for a surjective isometry} $$ \phi_f(e) = \phi_f({-w}) + P_0(\phi_f ({-w})) (\phi_f(e)) =- \phi_f({w}) + P_0(\phi_f ({w})) (\phi_f(e))$$ $$ = -\hat{w} + P_0(\hat{w}) (\phi_f(e)) .$$ We have therefore shown that for each minimal partial isometry $\hat{w}$ in $B^{**}$ such that $\phi_f(-e) = \hat{w} + P_0(\hat{w}) (\phi_f(-e))$ we have $- \phi_f(e) =\hat{w} + P_0(\hat{w}) (-\phi_f(e)).$ Finally, Proposition \ref{p technical} implies that $- \phi_f(e) =\phi_f(-e)+ P_0(\phi_f(-e)) (-\phi_f(e)).$ Replacing $e$ with $-e$ we have $- \phi_f(- e) =\phi_f(e)+ P_0(\phi_f(e)) (-\phi_f(-e)),$ and thus  $- \phi_f(- e) =  \phi_f( e).$\smallskip

The final statement is a consequence of the previous fact and Theorem \ref{t A}.
\end{proof}

We can continue now developing the multiple consequences which can be derived from Theorems \ref{t first correspondence between faces and compact partial isometries in the bidual for a surjective isometry}, \ref{t A} and \ref{t C}. The goal in our minds is to generalize \cite[Theorem 2.7]{FerPe17b} for surjective isometries between the unit spheres of two arbitrary von Neumann algebras.

\begin{proposition}\label{p algebraic elements} Let $f: S(A) \to S(B)$ be a surjective isometry between the unit spheres of two C$^*$-algebras. Then the following statements hold:\begin{enumerate}[$(a)$]\item For each non-zero partial isometry $v$ in $A$, the surjective real linear isometry $$T_v : (1 - vv^*) A  (1 - vv^*)\to  {(1 - f(v)f(v)^*) B (1 - f(v)^*f(v))}$$ given by Theorem \ref{t A} satisfies $f(e) = T_v(e),$ for every non-zero partial isometry $e\in (1 - vv^*) A  (1 - v^* v)$;
\item Let $w_1,\ldots,w_n$ be mutually orthogonal non-zero partial isometries in $A$, and let $\lambda_1,\ldots,\lambda_n$ be positive real numbers with $1=\lambda_1\geq\max\{\lambda_j\}$. Then $$f\left(\sum_{j=1}^n \lambda_j w_j\right) = \sum_{j=1}^n \lambda_j f\left(w_j\right);$$
\item Suppose $v,w$ are mutually orthogonal non-zero partial isometries in $A$ then $T_{v} (x) = T_{w} (x)$ for every $x\in \{v\}^{\perp} \cap \{w\}^{\perp}$;
\item If $A=M$ is a von Neumann algebra, for each non-zero partial isometry $v$ in $A$ we have $f(x) = T_v (x)$ for every $x\in S({(1 - vv^*) A  (1 - v^* v)})$.
\end{enumerate}
\end{proposition}

\begin{proof} $(a)$ Let $e,v$ be non-zero partial isometries in $A$ with $e\perp v$, and let $T_v$ and $ T_{\pm e} $ be the corresponding surjective real linear isometries  given by Theorem \ref{t A}. Combining Theorems \ref{t A} and \ref{t C} we have $$\pm f( e)+ T_{\pm e} (v) =f(\pm e)+ T_{\pm e} (v) = f(v\pm e) = f(v) \pm T_v (e),$$ where $T_v (e)\perp f(v),$ and $T_{\pm e} (v)\perp f(\pm e)= \pm f( e)$. Adding both expressions we get $$\displaystyle f(v) =\frac{ T_{e} (v) + T_{-e} (v)}{2}\perp f(e).$$ The facts $f(e+v) = f(e) + T_e (v) = f(v) +T_v(e)$ (see Theorem \ref{t A}) and $f(e)\perp f(v)$ prove that $$f(e+v) = f(e) +f(v) + P_0(f(e) +f(v)) (f(e+v)).$$ The mapping $f^{-1}$ satisfies the same property and hence $$f^{-1} (f(e)+f(v)) = f^{-1} f(e) + f^{-1} f(v) + P_0(f^{-1} f(e) +f^{-1} f(v)) (f^{-1} f(e+v))$$ $$ =e+v + P_0(e+v) (e+v)= e+v,$$ which gives $f(e+v)=f(e)+f(v)$, and in particular $f(e) = T_v(e)$ and $f(v) = T_e (v)$.\smallskip

$(b)$ Under the assumptions of the statement, it follows from $(a)$ and Theorem \ref{t A} that $$f\left(\sum_{j=1}^n \lambda_j w_j\right) = f(w_1) + T_{w_1} \left(\sum_{j=2}^n \lambda_j w_j\right) = f(w_1) +\sum_{j=2}^n \lambda_j T_{w_1} \left( w_j\right)= \sum_{j=1}^n \lambda_j f\left(w_j\right);$$

$(c)$ Let us take $x\in \{v\}^{\perp} \cap \{w\}^{\perp}$. By Theorem \ref{t A} and $(a)$ we deduce that $$f(v) +f (w) + T_{v} (x)= f(v) + T_{v} (w) + T_{v} (x)=f(v) +T_{v} (w+x) $$ $$= f(v+w+x) = f(w) + T_{w} (v+x)= f(w) + T_{w} (v) + T_{w} (x)= f(w) + f (v) + T_{w} (x),$$ which gives the desired equality.\smallskip

$(d)$ We recall that in a von Neumann algebra $M$ every hermitian element can be approximated in norm by finite real linear combinations of mutually orthogonal projections in $M$ (compare \cite[Theorem 1.11.3]{S}). For an arbitrary element $x\in M$ we consider its polar decomposition $x = v |x|$ with $v$ a partial isometry in $M$. Considering the von Neumann algebra $vv^* M vv^*$, it follows that $|x|$ can be approximated in norm by a finite real linear combination of mutually orthogonal projections $q_1,\ldots, q_m$ in $vv^* M vv^*$. We observe that $x$ can be approximated in norm by a finite real linear combination of $vq_1,\ldots, vq_m,$ with positive coefficient, and the latter are mutually orthogonal partial isometries in $M$ (compare \cite[Theorem 3.2]{Harris81} or \cite[Lemma 3.11]{Horn87} for a more general conclusion). Having in mind this property, we deduce from $(a)$ and $(b)$ that $f(x) = T_v (x)$ for each non-zero partial isometry $v$ in $M$ and every $x$ in $S({(1 - vv^*) M  (1 - v^* v)})=S(\{v\}^{\perp})=S(M_0(v))$.
\end{proof}

\section{A synthesis through convex combinations of unitary elements}

We begin this section by reviewing some results on surjective real linear isometries from a C$^*$-algebra into a JB$^*$-triple. We refer to \cite{Ka83}, \cite{FriRu85} and \cite{Da} and references therein for the definition and basic results on JB$^*$-triples. \smallskip

Let $A$ be a JB$^*$-algebra regarded as a real JB$^*$-triple in the sense employed in \cite{FerMarPe}. It is shown in the proof of \cite[Corollary 3.4]{FerMarPe} that $A^{**}$ contains no non-trivial real or complex rank one Cartan factors, and hence by \cite[Theorem 3.2]{FerMarPe} every surjective real linear isometry from $A$ into a real JB$^*$-triple is a triple isomorphism. This fact is gathered in the next result.

\begin{theorem}\label{t FerMarPe}\cite[Theorem 3.2 and Corollary 3.4]{FerMarPe} Let $T: A\to E$ be a surjective real linear isometry from a JB$^*$-algebra onto a JB$^*$-triple. Then $T$ preserves triple products. $\hfill\Box$
\end{theorem}

The subgroup $\mathcal{U}(A)$ of all unitary elements in a unital C$^*$-algebra $A$ is a special subset of the unit sphere of $A$. This subgroup encrypts a lot of information about $A$. In the proof of \cite[Theorem 4.12]{Tan2017b}, R. Tanaka proves that given a surjective isometry $f : S(M)\to S(N)$, where $M$ and $N$ are finite von Neumann algebras then $f$ preserves the unitary groups, that is, $f(\mathcal{U}(M)) = \mathcal{U}(N)$. As remarked in \cite[comments after Theorem 4.12]{Tan2017b}, the arguments applied to prove the previous equality are strongly based on the fact that $M$ and $N$ are finite, and hence every extreme point of their closed unit balls is a unitary element. We can extend now Tanaka's result for general von Neumann algebras.

\begin{theorem}\label{t Tanaka unitaries general vN} Let $f : S(M)\to S(N)$ be a surjective isometry between the unit spheres of two von Neumann algebras.  Then $f(\mathcal{U}(M)) = \mathcal{U}(N)$.
\end{theorem}

\begin{proof} Let us take $u\in \mathcal{U}(M)$. Theorem \ref{t A} implies that $f(u)$ is a partial isometry in $N$, and furthermore, the spaces $(1 - f(u)f(u)^*)N (1 - f(u)^*f(u))= N_0 (f(u))$ and $(1 - uu^*) M  (1 - u^*u)=M_0(u)=\{0\}$ are isometrically isomorphic as real Banach spaces, thus ${(1 - f(u)f(u)^*) N (1 - f(u)^*f(u))} =\{0\}$, and hence $f(u)$ is a maximal or complete partial isometry in $N$.\smallskip

We know from the above that $N = N_2 (f(u)) \oplus N_1 (f(u))$. If we prove that $N_1 (f(u))=\{0\}$ then $f(u)\in \mathcal{U}(N)$ as desired.\smallskip

If $u$ is a minimal partial isometry, then $M = \mathbb{C} u$ and hence the conclusion trivially follows from Tanaka's results \cite[Theorem 4.12]{Tan2017b} or \cite[Theorem 2.4]{Tan2017}. We can therefore assume the existence of two orthogonal non-zero partial isometries $u_1,u_2$ in $M$ such that $u = u_1 + u_2$. Proposition \ref{p algebraic elements} guarantees that $f(u) = f(u_1) + f(u_2)$, where $f(u_1)$ and $f(u_2)$ are orthogonal partial isometries. It is well known that $$ N_1 (f(u)) = N_1 (f(u_1))\cap N_0 (f(u_2)) \oplus N_0 (f(u_1))\cap N_1 (f(u_2)).$$

For each $j=1,2$, let $T_{u_j}: M_0 (u_j)\to N_0 (f(u_j))$ be the surjective real linear isometry given by Theorem \ref{t A}. It should be remarked that $M_0 (u_j)$ (respectively, $N_0 (f(u_j))$) need not be a von Neumann subalgebra of $M$ (respectively, of $N$), however it is always a JB$^*$-subtriple of $M$ (respectively, of $N$).\smallskip

Take an arbitrary non-zero partial isometry $w$ in $N_1 (f(u_1))\cap N_0 (f(u_2))$. By Theorem \ref{t A}, there exists a non-zero partial isometry $e$ in $M$ such that $f(e) =w$. Since $f(e) =w\in N_0 (f(u_2)),$ an application of Proposition \ref{p algebraic elements} assures that $e\in M_0 (u_2)$. On the other hand, $f(e) =w\in N_1 (f(u_1)) \cap N_0 (f(u_2)),$ $f(u_1)\in N_0 (f(u_2))$ and $T_{u_2}$ is a surjective real linear isometry. Since $M_0(u_2) = M_2(u_1)$ is a JB$^*$-algebra, Theorem \ref{t FerMarPe} guarantees that $T_{u_2}$ is a triple homomorphism. Combining these facts with Proposition \ref{p algebraic elements} we get $$ \frac12 T_{u_2} (e) = \frac12 f(e) = \{f(u_1),f(u_1),f(e)\} $$ $$= \{T_{u_2}(u_1),T_{u_2}(u_1),T_{u_2}(e)\} =T_{u_2}  \{u_1,u_1,e\}, $$ and hence $\{u_1,u_1,e\} = \frac12 e.$ Therefore $e\in M_1 (u_1)\cap M_0(u_2)=\{0\},$ because $u= u_1+u_2$ is a unitary in $M$. This contradicts that $f(e) = w$ is a non-zero partial isometry in $N_1 (f(u_1))\cap N_0 (f(u_2))$.\smallskip

We have shown that every partial isometry in $N_1 (f(u_1))\cap N_0 (f(u_2))$ is zero.
An argument similar to that given in the proof of Proposition \ref{p algebraic elements}$(d)$ assures that the elements in the JBW$^*$-subtriple $N_1 (f(u_1))\cap N_0 (f(u_2))$ which are finite linear combinations of mutually orthogonal partial isometries in $N_1 (f(u_1))\cap N_0 (f(u_2))$ are norm dense in this subtriple (compare  \cite[Theorem 3.2]{Harris81} or \cite[Lemma 3.11]{Horn87}). We can therefore conclude that $N_1 (f(u_1))\cap N_0 (f(u_2))=\{0\}$.\smallskip

Similar arguments to those given in the previous paragraphs show that $$N_0 (f(u_1))\cap N_1 (f(u_2))=\{0\},$$ and thus $N_1 (f(u))=\{0\},$ which concludes the proof. \end{proof}

We can now present a complete solution to Tignley's problem in the case of surjective isometries between the unit spheres of two arbitrary von Neumann algebras. The arguments here make use of a significant recent Theorem due to O. Hatori and L. Moln{\'a}r asserting that every surjective isometry between the unitary groups of two von Neumann algebras extends to a surjective real linear isometry between the algebras \cite[Corollary 3]{HatMol2014}.

\begin{theorem}\label{t Tingley von Neumann} Let $f : S(M)\to S(N)$ be a surjective isometry between the unit spheres of two von Neumann algebras. Then there exists a surjective real linear isometry $T: M\to N$ whose restriction to $S(M)$ is $f$. More precisely, there is a central projection $p$ in $N$ and a Jordan $^*$-isomorphism $J : M\to N$ such that defining $T : M\to N$ by $T(x) = f(1) \left( p J (x) + (1-p) J (x)^*\right)$ {\rm(}$x\in M${\rm)}, then $T$ is a surjective real linear isometry and $T|_{S(M)} = f$.
\end{theorem}

\begin{proof} We claim that we can reduce to the case in which $M$ is a factor. Otherwise, there exists a non-trivial central projection $p\in M$. In this case we can apply a procedure coined in \cite{PeTan16}. Let $T_p: M (1-p) \to (1-f(p)f(p)^*)N (1-f(p)^* f(p))$ and $T_{1-p} : M p \to (1-f(1-p)f(1-p)^*)N (1-f(1-p)^* f(1-p))$ be the surjective real linear isometries isometries given by Theorem \ref{t A}. We define a bounded real linear operator $T: M\to N$ given by $T(x) := T_p ((1-p) x) + T_{1-p} (x p)$ ($x\in M$). Let us take an arbitrary partial isometry $e$ in $M\backslash\{0\}$. The elements $e p $ and $e (1-p)$ are (orthogonal) partial isometries in $Mp $ and $M(1-p)$, respectively. If both are non-zero, then Proposition \ref{p algebraic elements}$(b)$ and $(d)$ implies that $$f(e) = f(e p ) + f(e (1-p)) = T_{1-p} (e p) +  T_p ((1-p) e) = T (e).$$ If $ep =0$ (respectively, $e (1-p)=0$), then $T(e) = T_p ((1-p) e) = f(e (1-p))= f(e)$ (respectively, $T(e) = T_{1-p} (p e) = f(e p) = f(e)$). Therefore $T$ and $f$ coincide on every non-zero partial isometry in $M$. Proposition \ref{p algebraic elements} assures that $T(x) = f(x)$ for every $x\in S(M)$.\smallskip

We therefore assume that $M$ and $N$ are factors. We shall distinguish several cases.\smallskip

Suppose first that $M$ is finite. Let us observe that we are not assuming that $N$ is finite. We recall that a von Neumann algebra $M$ is finite if and only if every extreme point of its closed unit ball is a unitary element (see \cite[Corollary 3]{Kad51} and \cite[Th{\'e}or\`{e}m 1, page 288]{Dix69}). Theorem \ref{t all faces von Neumann}$(b)$ (see \cite{FerGarPeVill17}) implies that $f(\partial_e (\mathcal{B}_M)) = \partial_e (\mathcal{B}_N)$, and Theorem \ref{t Tanaka unitaries general vN} assures that $f( \mathcal{U} (M)) = \mathcal{U} (N)$. Since  $\mathcal{U} (M)= \partial_e (\mathcal{B}_M),$ we conclude that $\partial_e (\mathcal{B}_N) = \mathcal{U} (N)$, and hence $N$ is finite too. Therefore $f:S(M)\to S(N)$ is a surjective isometry between the unit spheres of to finite factors. Theorem 4.12 in \cite{Tan2017b} proves the existence of a surjective real linear isometry $T$ from $M$ to $N$ whose restriction to $S(M)$ is $f$.\smallskip

Suppose next that $M$ is an infinite type I factor, that is, $M =B(H)$ where $H$ is an infinite-dimensional Hilbert space (see \cite[Corollary 5.5.8]{Ped}). We recall that a factor is of type I if and only if it contains a minimal projection. Let $p$ be a minimal projection in $M$. Proposition \ref{p algebraic elements}$(a)$ and $(b)$ guarantees that $f(p)$ is a minimal partial isometry in $N$. Consequently, $N$ contains a minimal projection, and hence $N$ is a type I factor. Therefore $f$ is a surjective isometry between the unit spheres of two type I von Neumann factors. Theorem 3.2 in \cite{FerPe17b} gives the desired conclusion.\smallskip

Finally we assume that $M$ is a type II$_{\infty}$ factor or a type III factor. By Theorem \ref{t Tanaka unitaries general vN}  we have $f( \mathcal{U} (M)) = \mathcal{U} (N)$. In particular $f(1)$ is a unitary in $N$. Therefore the mapping $f|_{ \mathcal{U} (M)} :  \mathcal{U} (M) \to  \mathcal{U} (N)$ is a surjective isometry. By \cite[Corollary 3]{HatMol2014} there is a central projection $p$ in $N$ and a Jordan $^*$-isomorphism $J : M\to N$ such that \begin{equation}\label{eq agree on unitary by HatoriMolnar} f(u) = f(1) \left( p J (u) + (1-p) J (u)^*\right),
\end{equation} for all $u\in \mathcal{U} (M) $. The mapping $T : M\to N$ defined by $$T(x) = f(1) \left( p J (x) + (1-p) J (x)^*\right)\ \  (x\in M)$$ is a surjective real linear isometry and $T(u) =f(u)$ for all $u\in \mathcal{U} (M) $. Since $N$ is factor, we can assume that $p=0$ or $p=1$.\smallskip

We shall prove that $T(x) = f(x)$ for all $x\in S(M)$. We have already commented in several arguments before that every element in $S(M)$ can be approximated by a finite linear combination $\displaystyle \sum_{j=1}^{m} \lambda_j e_j$ of mutually orthogonal (non-zero) partial isometries $e_1,\ldots, e_m$ in $M$ and $\lambda_1=1\geq \lambda_2,\ldots,\lambda_m$ in $\mathbb{R}^+$. Therefore, by Proposition \ref{p algebraic elements}, it will enough to prove that $T(v) = f(v)$ for every non-zero partial isometry $v$ in $M$.\smallskip

Let $e$ be a maximal partial isometry in $M$, that is, an element in $\partial_{e} (\mathcal{B}_M)$. Since $M$ is a factor we may assume that $e^* e =1$. By the Halving lemma (see \cite[Lemma 6.3.3]{KadRingII1986}), there exist orthogonal projections $q_1,q_2$ in $M$ such that $1 = q_1 + q_2$ and $q_1 \sim 1 \sim q_2$, that is, $1,q_1$ and $q_2$ are pairwise Murray-von Neumann equivalent. Let us set $e_1 = e q_1$ and $e_2 = e q_2$. It is easy to check that $e_1$ and $e_2$ are mutually orthogonal partial isometries with $e = e_1 + e_2$. Furthermore, since $e_j^*e_j = q_j$ and $e_j e_j^* = e q_j e^*$, we can deduce that $1\sim q_1\sim e_1 e_1^* \leq 1- e_2 e_2^* \leq 1,$ and $1\sim q_2\sim e_2 e_2^* \leq 1- e_1 e_1^* \leq 1$, and hence $1- e_2 e_2^* \sim 1\sim 1- e_1 e_1^*$ (see \cite[Proposition 6.2.4]{KadRingII1986}). Clearly, $1- e_1^* e_1 = 1-q_1 = q_2 \sim 1$ and $1- e_2^* e_2 = 1-q_2 = q_1 \sim 1.$ Let $\mathcal{G}({M})$ denote the group of invertible elements in $M$. By \cite[Theorem 2.2]{Ols1989} the formula $$\hbox{dist} (a,\mathcal{G}(M)) = \inf \{\lambda: E[0,\lambda]\sim F[0,\lambda]\},$$ holds for each element $a\in M$, where $\{ E[0,\lambda] \}$ and $\{ F[0,\lambda] \}$ are the spectral resolutions of $|a|$ and $|a^*|$, respectively, and $E[0,\lambda]\sim F[0,\lambda]$. Applying the above formula to $e_1$ and $e_2$, we get  $\hbox{dist} (e_1,\mathcal{G}(M)) =\hbox{dist} (e_2,\mathcal{G}(M)) = 0.$ Now, combining Lemma 2.1 and Theorem 2.10 in \cite{OlsPed1986} we deduce that $e_1$ and $e_2$ can be written as a convex combination of at most three unitary elements $u^{j}_1,u^{j}_2,u^{j}_3$ in $M$, that is, $\displaystyle e_j = \sum_{k=1}^{3} t_k u^{j}_k$ with $t_k \in \mathbb{R}^+_0$ and $\displaystyle \sum_{k=1}^{3} t_k=1$  ($j=1,2$).\smallskip

Let $F_{e_j} = e_{j} + (1-e_{j}e_{j}^*) \mathcal{B}_{M} (1-e_{j}^*e_{j})$ denote the norm closed face of $\mathcal{B}_{M}$ containing $e_{j}$. Since $\displaystyle \sum_{k=1}^{3} t_k u^{j}_k=e_{j} \in F_{e_j}$ and the latter is a face of $\mathcal{B}_{M}$, we have $u^{j}_k \in F_{e_j}$ for all $k\in \{1,2,3\}$ with $t_k\neq 0$. By Theorem \ref{t A} the mapping $f|_{F_{e_{j}}}$ is a real affine function, therefore $$f(e_{j}) = \sum_{k=1}^{3} t_k f(u_k) = \hbox{ (by \eqref{eq agree on unitary by HatoriMolnar}) } = \sum_{k=1}^{3} t_k T(u_k)  = T(e_{j}).$$

Now, Proposition \ref{p algebraic elements}$(b)$ and what is proved in the previous paragraph show that $$f(e) = f(e_1) + f(e_2) = T(e_1) +T(e_2) = T(e),$$ for every maximal partial isometry $e$ in $M$.\smallskip

Finally, let $v$ be an arbitrary non-zero partial isometry in $M$. Let us find a maximal partial isometry $e$ satisfying $v\leq e$. We set $\widetilde{e} = v- (e-v)\in \partial_{e} (\mathcal{B}_{M}).$ If we combine Proposition \ref{p algebraic elements}$(b)$, Theorem \ref{t C}, and the previous conclusion we get $$f(v) + f(e-v) = f(e) = T(e) = T(v) + T(e-v),$$ and $$f(v) - f(e-v) = f(\widetilde{e}) = T(\widetilde{e}) = T(v) - T(e-v),$$ and hence $T(v) = f(v)$. We have therefore proved that $T(v) = f(v)$ for every non-zero partial isometry $v$ in $M$, which concludes the proof.
\end{proof}

\begin{corollary}\label{c factors} Let $f : S(M)\to S(N)$ be a surjective isometry between the unit spheres of two von Neumann algebras. Let us assume that $M$ is a factor. Then there exists a Jordan $^*$-isomorphism $J : M\to N$ such that exactly one of the following statements holds:\begin{enumerate}[$(a)$]\item $f(x) = f(1) J (x),$ for all $x\in S(M)$;
\item $f(x) = f(1) J (x^*),$ for all $x\in S(M)$. $\hfill\Box$
\end{enumerate}
\end{corollary}

\medskip\medskip

\textbf{Acknowledgements} Authors partially supported by the Spanish Ministry of Economy and Competitiveness (MINECO) and European Regional Development Fund project no. MTM2014-58984-P and Junta de Andaluc\'{\i}a grant FQM375.

\end{document}